\documentclass[a4paper,10pt]{amsart}

\usepackage{amssymb,amscd,amsmath}

\theoremstyle{plain}

\newtheorem{thm}{Theorem}[section]

\newtheorem{lem}[thm]{Lemma}

\newtheorem{obs}[thm]{Observation}

\theoremstyle{definition}

\newtheorem{prob}{Problem}[section]

\theoremstyle{remark}

\newcommand{\forme}[1]{}

\title{Restrained Italian domination in trees}

\author[Kim]{Kijung Kim}
\address{Department of Mathematics, Pusan National University, Busan 46241, Republic of Korea}
\email{knukkj@pusan.ac.kr}

\date{\today}
\begin{document}

\begin{abstract}

Let $G=(V,E)$ be a graph.
A subset $D$ of $V$ is a \textit{restrained dominating set} if every vertex in $V \setminus D$ is adjacent to a vertex in $D$ and to a vertex in $V \setminus D$.
The \textit{restrained domination number}, denoted by $\gamma_r(G)$, is the smallest cardinality of a restrained dominating set of $G$.
A function $f : V \rightarrow \{0, 1, 2\}$ is a \textit{restrained Italian dominating function} on $G$ if
(i) for each vertex $v \in V$ for which $f(v)=0$, it holds that $\sum_{u \in N_G(v)} f(u) \geq 2$,
(ii) the subgraph induced by $\{v \in V \mid f(v)=0 \}$ has no isolated vertices.
The \textit{restrained Italian domination number}, denoted by $\gamma_{rI}(G)$, is the minimum weight taken over all restrained Italian dominating functions of $G$.
It is known that $\gamma_r(G) \leq \gamma_{rI}(G) \leq 2\gamma_r(G)$ for any graph $G$.
In this paper, we characterize the trees $T$ for which $\gamma_r(T) = \gamma_{rI}(T)$, and we also characterize
the trees $T$ for which $\gamma_{rI}(T) = 2\gamma_r(T)$.

\bigskip

\end{abstract}

\maketitle

\insert\footins{\footnotesize
This research was supported by Basic Science Research Program through the National Research Foundation of Korea funded by the Ministry of Education (2020R1I1A1A01055403).}

\section{Introduction and Terminology}\label{sec:intro}

Let $G=(V,E)$ be a finite simple graph with vertex set $V=V(G)$ and edge set $E=E(G)$.
The \textit{open neighborhood} of $v \in V(G)$ is the set $N_G(v) = \{ u \in V(G) \mid uv \in E(G)\}$
and the \textit{closed neighborhood} of $v \in V(G)$ is the set $N_G[v]:= N_G(v) \cup \{v\}$.
A subset $D$ of $V(G)$ is a \textit{dominating set} if every vertex in $V(G) \setminus D$ is adjacent to a vertex in $D$.
The \textit{domination number} of $G$, denoted by $\gamma(G)$, is the minimum cardinality of a dominating set in $G$.
A dominating set with the cardinality $\gamma(G)$ is called a \textit{$\gamma(G)$-set}.

In \cite{DHHLM}, Domke et al. gave the formal definition of restrained domination.
A subset $S$ of $V(G)$ is a \textit{restrained dominating set} (RDS) if every vertex in $V(G) \setminus S$ is adjacent to a vertex in $S$
and another vertex in $V(G) \setminus S$.
The \textit{restrained domination number} of $G$, denoted by $\gamma_r(G)$, is the minimum cardinality of a restrained dominating set in $G$.
A restrained dominating set with the cardinality $\gamma_r(G)$ is called a \textit{$\gamma_r(G)$-set}.
As explained in \cite{DHHLM}, there is one possible application of the concept of restrained domination.
Each vertex in a RDS $S$ represents a guard and each vertex in $V(G) \setminus S$ represents a prisoner.
Each prisoner must be observed by at least one guard and every prisoner must be seen by at least one
other prisoner to protect the rights of prisoners.
To be cost effective, it is desirable to place as few guards as possible.

A function $f : V(G) \rightarrow \{0, 1, 2\}$ is an \textit{Italian dominating function} on $G$ if
for each vertex $v \in V(G)$ for which $f(v)=0$, it holds that $\sum_{u \in N_G(v)} f(u) \geq 2$.
In \cite{SAMM}, Samadi et al. introduced the concept of restrained Italian domination
as a variant of Italian dominating function.
An Italian dominating function $f : V \rightarrow \{0, 1, 2\}$ is a \textit{restrained Italian dominating function} (RIDF) on $G$ if
the subgraph induced by $\{v \in V \mid f(v)=0 \}$ has no isolated vertices.
A RIDF $f$ gives an ordered partition $(V_0, V_1, V_2)$ (or $(V_0^f, V_1^f, V_2^f)$ to refer to $f$) of $V(G)$, where $V_i := \{ v \in V(G) \mid f(v)=i \}$.
The weight of a RIDF $f$ is $\omega(f):=\sum_{v \in V} f(v)$.
The \textit{restrained Italian domination number}, denoted by $\gamma_{rI}(G)$, is the minimum weight taken over all restrained Italian dominating functions of $G$.
A \textit{$\gamma_{rI}(G)$-function} is a RIDF on $G$ with weight $\gamma_{rI}(G)$.

As noted in \cite[Proposition 3.3]{SAMM}, it holds that $\gamma_r(G) \leq \gamma_{rI}(G) \leq 2\gamma_r(G)$ for any graph $G$.
We define a tree $T$ to be a \textit{$(\gamma_r, \gamma_{rI})$-tree} if $\gamma_r(T) = \gamma_{rI}(T)$.
We define a tree $T$ to be a \textit{restrained Italian tree} if $\gamma_{rI}(T) = 2\gamma_r(T)$.
In this paper, we characterize $(\gamma_r, \gamma_{rI})$-trees and restrained Italian trees.

The rest of this section, we present some necessary terminology and notation.
For terminology and notation on graph theory not given here, the reader is referred to \cite{BM, W}.
The \textit{degree} of $v \in V(G)$ is defined as the cardinality of $N_G(v)$, denoted by $deg_G(v)$.
A \textit{diametral path} of $G$ is a path with the length which equals the diameter of $G$.
A subset $S$ of $V(G)$ is a \textit{packing} in $G$ if the vertices of $S$ are pairwise at distance at least three apart in $G$.
The \textit{packing number} of $G$, denoted by $\rho(G)$, is the maximum cardinality of a packing in $G$.
A packing with the cardinality $\rho(G)$ is called a \textit{$\rho(G)$-set}.

Let $T$ be a (rooted) tree.
A \textit{leaf} of $T$ is a vertex of degree one.
A \textit{stem} (or \textit{support vertex}) is a vertex adjacent to a leaf.
A \textit{weak stem} is a stem that is adjacent to exactly one leaf.
For a vertex $v$ in a rooted tree,
we let $C(v)$ and $D(v)$ denote the set of children and descendants, respectively, of $v$.
The subtree induced by $D(v) \cup \{v\}$ is denoted by $T_v$.
We write $K_{1,n-1}$ for the \textit{star} of order $n \geq 3$.
The \textit{double star} $DS_{p,q}$, where $p, q \geq 1$, is the graph obtained
by joining the centers of two stars $K_{1,p}$ and $K_{1,q}$.
A \textit{healthy spider} $S_{t,t}$ is the graph from a star $K_{1,t}$ by subdividing each edges of $K_{1,t}$.
For two graph $G$ and $H$, if $G$ is isomorphic to $H$, we denote it by $G \cong H$.
For a graph $G$ and its subgraph $S$, $G - S$ denotes the subgraph of $G$ induced by $V(G) \setminus V(S)$.

\section{$(\gamma_r, \gamma_{rI})$-trees}\label{sec:rI=r}

In this section, we characterize the trees for which $\gamma_{rI}(T) = \gamma_r(T)$.
First, we introduce a family $\mathcal{H}$ of trees that can be obtained from a sequence $T_1, T_2, \dotsc, T_m$ $(m \geq 1)$
of trees such that $T_1$ is a double star $DS_{l,n}$ $(l,n \geq 2)$, and if $m \geq 2$, $T_{i+1}$ can be obtained recursively from $T_i$ by one of the following operations for $1 \leq i \leq m-1$.

Define
\[LV(T_i) =\{ v \in V(T_i) \mid v ~\text{is a leaf of}~ T_j ~\text{for some}~ j \leq i\}\]
and
\[SV(T_i) =\{ v \in V(T_i) \mid v ~\text{is a stem of}~ T_j ~\text{for some}~ j \leq i\}.\]
Note that $V(T_i) = LV(T_i) \cup SV(T_i)$.

\vskip5pt
\textbf{Operation} $\mathcal{O}_1$.
If $x \in LV(T_i)$, then $\mathcal{O}_1$ adds a double star $DS_{r,s}$ with a center $u$ and joins $u$ to $x$ to produce $T_{i+1}$,
where $s \geq 2$ and $u$ has $r$ leaves.

\vskip5pt
\textbf{Operation} $\mathcal{O}_2$.
If $x \in SV(T_i)$, then $\mathcal{O}_2$ adds a star $K_{1,t}$ with the center $u$ and joins $u$ to $x$ to produce $T_{i+1}$.

The following is obtained by the induction.

\begin{obs}\label{obs:remark}
With the previous notation, the following holds.
\begin{enumerate}
\item $LV(T_i)$ is a unique minimum RDS of $T_i$.
\item The subgraph induced by $SV(T_i)$ is a forest and each component has at least two vertices.
\end{enumerate}
\end{obs}

\begin{lem}\label{lem:op3}
If  $\gamma_r(T_i) = \gamma_{rI}(T_i)$ and $T_{i+1}$ is obtained from $T_i$ by operation $\mathcal{O}_1$,
then $\gamma_r(T_{i+1}) = \gamma_{rI}(T_{i+1})$.
\end{lem}
\begin{proof}
It follows from Observation \ref{obs:remark} that $\gamma_r(T_{i+1}) = \gamma_r(T_i) +r + s$.
Since every $\gamma_{rI}(T_{i})$-function can be extended to a RIDF of $T_{i+1}$,
we have $\gamma_{rI}(T_{i+1}) \leq \gamma_{rI}(T_i) +r + s$.


We verify $\gamma_{rI}(T_{i+1}) = \gamma_{rI}(T_i) +r + s$.
Let $g$ be a $\gamma_{rI}(T_{i+1})$-function.
If $g(x)=0$, then $g(y)=1$ for each $y \in N_{T_{i+1}}(x)$.
This implies that $\gamma_{rI}(T_{i+1}) \geq \gamma_{rI}(T_i) +r + s +2$, a contradiction.
Thus, we have $g(x)=1$.
It is easy to see that $g|_{V(T_i)}$ is a RIDF.
So, we have $\gamma_{rI}(T_i) \leq \gamma_{rI}(T_{i+1}) -r -s$.

Thus, it follows from $\gamma_r(T_i) = \gamma_{rI}(T_i)$ that $\gamma_r(T_{i+1}) = \gamma_{rI}(T_{i+1})$.
\end{proof}

\begin{lem}\label{lem:op4}
If  $\gamma_r(T_i) = \gamma_{rI}(T_i)$ and $T_{i+1}$ is obtained from $T_i$ by operation $\mathcal{O}_2$,
then $\gamma_r(T_{i+1}) = \gamma_{rI}(T_{i+1})$.
\end{lem}
\begin{proof}
It follows from Observation \ref{obs:remark} that $\gamma_r(T_{i+1}) = \gamma_r(T_i) +t$.
Since every $\gamma_{rI}(T_{i})$-function can be extended to a RIDF of $T_{i+1}$,
we have $\gamma_{rI}(T_{i+1}) \leq \gamma_{rI}(T_i) +t$.

We verify $\gamma_{rI}(T_{i+1}) = \gamma_{rI}(T_i) +t$.
Let $g$ be a $\gamma_{rI}(T_{i+1})$-function.
If $g(x)=1$, then $g(u)=1$.
This implies that $\gamma_{rI}(T_{i+1}) \geq \gamma_{rI}(T_i) +t +1$, a contradiction.
Thus, we have $g(x)=0$.
It is easy to see that $g|_{V(T_i)}$ is a RIDF.
So, we have $\gamma_{rI}(T_i) \leq \gamma_{rI}(T_{i+1}) -t$.

Thus, it follows from $\gamma_r(T_i) = \gamma_{rI}(T_i)$ that $\gamma_r(T_{i+1}) = \gamma_{rI}(T_{i+1})$.
\end{proof}

Now we are ready to prove our main theorem.

\begin{thm}\label{tree:low-inequality}
A tree $T$ of order $n \geq 3$ is a $(\gamma_r, \gamma_{rI})$-tree if and only if $T \in \mathcal{H} \cup \{K_{1,t} \mid t \geq 2\}$.
\end{thm}
\begin{proof}
First, we prove that if $T \in \mathcal{H} \cup \{K_{1,t} \mid t \geq 2\}$, then $\gamma_{rI}(T) = \gamma_r(T)$.
Clearly, $\gamma_{rI}(K_{1,t}) = \gamma_r(K_{1,t})$.
Assume that $T \in \mathcal{H}$.
Then there exist a sequence $T_1, T_2, \dotsc, T_m =T$ $(m \geq 1)$ such that $T_1$ is a double star $DS_{r,s}$,
and if $m \geq 2$, $T_{i+1}$ can be obtained recursively from $T_i$ by an operation $\mathcal{O}_1$ or $\mathcal{O}_2$
for $1 \leq i \leq m-1$.
We use induction on $m$.
Clearly, $\gamma_{rI}(T_1) = \gamma_r(T_1)$.
Suppose that the statement is true for any tree constructed by $m-1$ operations.
Let $T' = T_{m-1}$. By the induction hypothesis, $\gamma_{rI}(T') = \gamma_r(T')$.
It follows from Lemma \ref{lem:op3} or \ref{lem:op4} that $\gamma_{rI}(T) = \gamma_r(T)$.

Next, we prove that if $\gamma_{rI}(T) = \gamma_r(T)$, then $T \in \mathcal{H} \cup \{K_{1,t} \mid t \geq 2\}$.
We proceed by induction on the order $n$ of $T$ satisfying $\gamma_{rI}(T) = \gamma_r(T)$.
Suppose that $diam(T)=2$.
Then $T$ is a star and clearly $\gamma_{rI}(T) = \gamma_r(T)$ Thus, $T \in \{K_{1,t} \mid t \geq 2\}$.
Suppose that $diam(T)=3$.
Then $T \cong DS_{r,s}$ for $r, s \geq 2$.
In this case, $T$ can be obtained from $K_{1,r}$ by operation $\mathcal{O}_2$.
Hence, we may assume that $diam(T) \geq 4$.

Among all of diametrical paths in $T$, we choose $x_0x_1\dotsc x_d$ so that it maximizes the degree of $x_{d-1}$.
Root $T$ at $x_0$.
Let $g=(V_0^g, V_1^g, V_2^g)$ be a $\gamma_{rI}(T)$-function.

\vskip5pt
\textbf{Claim 1.} $V_2^g = \emptyset$ and $V_1^g$ is a RDS of $T$.

Since $V_1^g \cup V_2^g$ is a RDS of $T$, we have
\[\gamma_r(T) \leq |V_1^g \cup V_2^g| = |V_1^g| +|V_2^g| \leq |V_1^g| +2|V_2^g| = \gamma_{rI}(T).\]

Since $\gamma_{rI}(T) = \gamma_r(T)$, we must have the equality throughout the above inequality chain.
Thus, $V_2^g = \emptyset$ and $V_1^g$ is a RDS of $T$.

\vskip5pt
\textbf{Claim 2.} $deg_T(x_{d-1}) \geq 3$.

Suppose to the contrary that $deg_T(x_{d-1}) =2$.
Suppose that $deg_T(x_{d-2}) =2$.
In this case, $g(x_{d-1})=1$ for otherwise $g(x_{d-2})$ must be assigned the weight $1$
but this contradicts the fact that $V_1^g$ is a RDS of $T$.
By the same argument, we have $g(x_{d-2})=1$.
If $g(x_{d-3})=1$, then $V_1^g \setminus \{ x_{d-2}, x_{d-1} \}$ is a RDS with the cardinality less than $\gamma_r(T)$, a contradiction.
Thus, $g(x_{d-3})=0$ and $\sum_{x \in N_T(x_{d-3})} g(x) \geq 2$.
This implies that $V_1^g \setminus \{x_{d-2}\}$ is a RDS of $T$.
This is a contradiction.

Suppose that $deg_T(x_{d-2}) \geq 3$.
Then each $x \in N_T(x_{d-2}) \setminus \{x_{d-3}\}$ is either a leaf or a weak stem
by $deg_T(x_{d-1}) =2$ and the hypothesis about $deg_T(x_{d-1})$.
If $g(x_{d-2})=1$, then every vertex of $T_{x_{d-2}}$ has the weight $1$.
Let $M$ be the subset of $V_1^g$ obtained by removing $x_{d-2}$ and weak stems in $T_{x_{d-2}}$.
Now we have $g(x_{d-3})=0$ for otherwise $M$  is a RDS of $T$, a contradiction.
Since $\sum_{x \in N_T(x_{d-3})} g(x) \geq 2$, $M$ is a RDS of $T$, a contradiction.
This completes the proof of claim.

\vskip5pt
\textbf{Claim 3.} $deg_T(x_{d-2}) \geq 3$.

Suppose to the contrary that $deg_T(x_{d-2}) =2$.
Suppose that $g(x_{d-2})=0$.
Then $g(x_{d-3})= g(x_{d-1})=1$, since $\sum_{x \in N_T(x_{d-2})} g(x) \geq 2$.
This is a contradiction by Claim 1.

Suppose that $g(x_{d-2})=1$.
Then $g(x_{d-1})=1$.
Let $N$ be the subset of $V_1^g$ obtained by removing $x_{d-2}$ and $x_{d-1}$.
If $g(x_{d-3})=1$, then $N$ is a RDS of $T$, a contradiction.
Thus, we have $g(x_{d-3})=0$. But, since $\sum_{x \in N_T(x_{d-3})} g(x) \geq 2$,
$N \cup \{x_{d-2}\}$ is a RDS of $T$, a contradiction.
This completes the proof of claim.

\vskip5pt
We divide our consideration into two cases.

Case 1. $x_{d-3} \in V_1^g$.
Then $x_{d-2}, x_{d-1} \in V_0^g$ for otherwise
the subset of $V_1^g$ obtained by removing $x_{d-2}$ and stems in $T_{x_{d-2}}$ is a RDS of $T$, a contradiction.

\vskip5pt
Subcase 1.1. $x_{d-2}$ has stems except for $x_{d-1}$.
Consider $T_{x_{d-1}} \cong K_{1,t}$ and let $T'= T - T_{x_{d-1}}$.
Since every $\gamma_{r}(T')$-set (resp., $\gamma_{rI}(T')$-function) can be extended to a RDS (resp., RIDF) of $T$,
$\gamma_{r}(T) \leq \gamma_{r}(T')+ t$ and $\gamma_{rI}(T) \leq \gamma_{rI}(T')+ t$.
Since $V_1^g \setminus C(x_{d-1})$ (resp., $g|_{V(T')}$) is a RDS (resp., RIDF) of $T'$,
$\gamma_{r}(T') \leq \gamma_{r}(T) - t$ and $\gamma_{rI}(T') \leq \gamma_{rI}(T) - t$.
Thus, it follows from $\gamma_{rI}(T) = \gamma_r(T)$ that $\gamma_{r}(T') = \gamma_{rI}(T')$.
Applying the inductive hypothesis to $T'$, $T' \in \mathcal{H}$.
By operation $\mathcal{O}_2$, we have $T \in \mathcal{H}$.

\vskip5pt
Subcase 1.2. $x_{d-2}$ has no stem except for $x_{d-1}$.
Consider $T_{x_{d-2}} \cong DS_{r,s}$ and let $T'= T - T_{x_{d-2}}$.
Since every $\gamma_{r}(T')$-set (resp., $\gamma_{rI}(T')$-function) can be extended to a RDS (resp., RIDF) of $T$,
we have $\gamma_{r}(T) \leq \gamma_{r}(T') +r + s$ and $\gamma_{rI}(T) \leq \gamma_{rI}(T') +r + s$.
Since $V_1^g \setminus D(x_{d-2})$ (resp., $g|_{V(T')}$) is a RDS (resp., RIDF) of $T'$,
$\gamma_{r}(T') \leq \gamma_{r}(T) -r -s$ and $\gamma_{rI}(T') \leq \gamma_{rI}(T) -r -s$.
Thus, it follows from $\gamma_{rI}(T) = \gamma_r(T)$ that $\gamma_{r}(T') = \gamma_{rI}(T')$.
Applying the inductive hypothesis to $T'$, $T' \in \mathcal{H}$.
By operation $\mathcal{O}_1$, we have $T \in \mathcal{H}$.

\vskip5pt
Case 2. $x_{d-3} \in V_0^g$.
Then $x_{d-2}, x_{d-1} \in V_0^g$ for otherwise every vertex in $T_{x_{d-2}}$ belongs to $V_1^g$.
Since $x_{d-3}$ is adjacent to at least one vertex not in $T_{x_{d-2}}$,
the subset of $V_1^g$ obtained by removing $x_{d-2}$ and stems in $T_{x_{d-2}}$ is a RDS of $T$, a contradiction.

Consider $T_{x_{d-1}} \cong K_{1,t}$ and let $T'= T - T_{x_{d-1}}$.
Since every $\gamma_{r}(T')$-set (resp., $\gamma_{rI}(T')$-function) can be extended to a RDS (resp., RIDF) of $T$,
$\gamma_{r}(T) \leq \gamma_{r}(T')+ t$ and $\gamma_{rI}(T) \leq \gamma_{rI}(T')+ t$.
Since $V_1^g \setminus C(x_{d-1})$ (resp., $g|_{V(T')}$) is a RDS (resp., RIDF) of $T'$,
$\gamma_{r}(T') \leq \gamma_{r}(T) - t$ and $\gamma_{rI}(T') \leq \gamma_{rI}(T) - t$.
Thus, it follows from $\gamma_{rI}(T) = \gamma_r(T)$ that $\gamma_{r}(T') = \gamma_{rI}(T')$.
Applying the inductive hypothesis to $T'$, $T' \in \mathcal{H}$.
By operation $\mathcal{O}_2$, we have $T \in \mathcal{H}$.
\end{proof}

\section{Restrained Italian trees}\label{sec:rI=2r}

In this section, we characterize the trees for which $\gamma_{rI}(T) = 2\gamma_r(T)$.
First, we introduce a family $\mathcal{F}$ of trees that can be obtained from a sequence $T_1, T_2, \dotsc, T_m$ $(m \geq 1)$
of trees such that $T_1$ is a path $P_4$, and if $m \geq 2$, $T_{i+1}$ can be obtained recursively from $T_i$ by one of the following operations
for $1 \leq i \leq m-1$.

Define
\[LV(T_i) =\{ v \in V(T_i) \mid v ~\text{is a leaf of}~ T_j ~\text{for some}~ j \leq i\}.\]

\vskip5pt
\textbf{Operation} $\mathcal{O}_1$.
If $x \in LV(T_i)$, then $\mathcal{O}_1$ adds a path $P_3$ with a leaf $u$ and joins $u$ to $x$ to produce $T_{i+1}$.

\vskip5pt
\textbf{Operation} $\mathcal{O}_2$.
If $x \in LV(T_i)$, then $\mathcal{O}_2$ adds a healthy spider $S_{t,t}$ with the center $u$ and joins $u$ to $x$ to produce $T_{i+1}$.

Since the family $\mathcal{F}$ is a subclass of $\mathcal{T}$ given in \cite[Lemma 3]{DHHS}, we can get the following result.

\begin{lem}\label{lem:global}
With the previous notation, the following properties hold.
\begin{enumerate}
\item $LV(T_i)$ is a packing.
\item Every $v \in V(T_i) \setminus LV(T_i)$ is adjacent to at least one vertex in $V(T_i) \setminus LV(T_i)$ and to exactly one vertex in $LV(T_i)$.
\item $LV(T_i)$ is a $\gamma(T_i)$-set.
\item $LV(T_i)$ is the unique $\gamma_r(T_i)$-set.
\item $LV(T_i)$ is the unique $\rho(T_i)$-set.
\end{enumerate}
\end{lem}

\begin{lem}\label{lem:key}
With the previous notation,
$(V(T_i) \setminus LV(T_i), \emptyset, LV(T_i))$ is a $\gamma_{rI}(T_i)$-function.
\end{lem}
\begin{proof}

We show that every RIDF has weight at least $2|LV(T_i)|$.
Let $f$ be a RIDF of $T_i$ and $P(T_i):=\{ N_{T_i}[v] \mid v \in LV(T_i)\}$.
It follows from Lemma \ref{lem:global} that $P(T_i)$ is a partition of $V(T_i)$.

We claim that $f(U) \geq 2$ for each $U \in P(T_i)$.
For a leaf $v \in V(T_i)$, clearly $f(N_{T_i}[v])=2$.
Suppose to the contrary that $f(N_{T_i}[u]) = 1$ for some $u \in LV(T_{i-1})$.
For $w \in N_{T_i}(u)$, if $f(w)=1$, then $u$ is not dominated, a contradiction.
Assume that $f(u)=1$.
From the construction of $T_i$, there exist at least one vertex $z \in N_{T_i}(u)$ such that $deg_{T_i}(z)=2$.
To dominate $z$, there must be one vertex with weight at least one.
This is not restrained, a contradiction.
Thus, $\gamma_{rI}(T_i) \geq 2|LV(T_i)|=2\gamma_r(T_i)$ and
clearly $(V(T_i) \setminus LV(T_i), \emptyset, LV(T_i))$ is a $\gamma_{rI}(T_i)$-function.
\end{proof}

Now we are ready to prove our main theorem.

\begin{thm}\label{tree:up-inequality}
A tree $T$ of order $n \geq 4$ is a restrained Italian tree if and only if $T \in \mathcal{F}$.
\end{thm}
\begin{proof}

The sufficiency follows from Lemmas \ref{lem:global} and \ref{lem:key}.
To prove the necessity,
we proceed by induction on the order $n$ of $T$ satisfying $\gamma_{rI}(T) = 2\gamma_r(T)$.
It suffices to consider trees with diameter at least three.
Suppose that $diam(T)=3$.
Then $T \cong DS_{r,s}$.
Since $\gamma_{r}(T) = \gamma_{rI}(T) = r+s$ for $r, s \geq 2$,
we have $DS_{1,1} \cong P_4 \in \mathcal{F}$.
Hence, we assume that $diam(T) \geq 4$ and $n \geq 5$.

Among all of diametrical paths in $T$, we choose $x_0x_1\dotsc x_d$ so that it maximizes the degree of $x_{d-1}$.
Root $T$ at $x_0$.
Let $D$ be a $\gamma_r(T)$-set and $g$ be a $\gamma_{rI}(T)$-function
defined by $g(v)=2$ for $v \in D$ and $g(u)=0$ for $u \in V(T) \setminus D$.

\vskip5pt
\textbf{Claim 1.} $x_{d-1} \not\in D$.

Suppose to the contrary that $x_{d-1} \in D$.
Since $x_d \in D$ and $\gamma_{rI}(T) = 2\gamma_r(T)$,
we can define a function $f : V(T) \rightarrow \{0,1,2\}$ by $f(x_d)=1$ and $f(x)=g(x)$ otherwise.
Then $f$ is a RIDF of $T$ with weight less than $\omega(g)$, a contradiction.

\vskip5pt
\textbf{Claim 2.} $deg_T(x_{d-1})=2$.

Suppose to the contrary that $deg_T(x_{d-1}) \geq 3$.
Then there exists at least one leaf $u \in N_T(x_{d-1}) \setminus \{x_d\}$.
Since $u, x_d \in D$ and $\gamma_{rI}(T) = 2\gamma_r(T)$,
we can define $f : V(T) \rightarrow \{0,1,2\}$ by $f(u)=f(x_d)=1$ and $f(x)=g(x)$ otherwise.
Then $f$ is a RIDF of $T$ with weight less than $\omega(g)$, a contradiction.

\vskip5pt
We divide our consideration into two cases.

Case 1. $deg_T(x_{d-2}) \geq 3$.
By Claim 2, each $x \in N_T(x_{d-2}) \setminus \{ x_{d-3}\}$ is either a leaf or a weak stem.
Now we show that $x_{d-2}$ has no leaf.
Suppose to the contrary that there exists a leaf $u \in N_T(x_{d-2})$.
Then $u \in D$. Note that $x_{d-1}, x_{d-2} \not\in D$.
If $x_{d-3} \in D$, then define $f : V(T) \rightarrow \{0,1,2\}$ by $f(u)=1$ and $f(x)=g(x)$ otherwise.
Clearly $f$ is a RIDF of $T$ with weight less than $\omega(g)$, a contradiction.
Suppose that $x_{d-3} \not\in D$.
Define $h : \rightarrow \{0,1,2\}$ by $h(u) =h(x_{d-1})= h(x_d) =1$ and $h(x)=g(x)$ otherwise.
Clearly $h$ is a RIDF of $T$ with weight less than $\omega(g)$, a contradiction.
Thus, $x_{d-2}$ has no leaf and so $T_{x_{d-2}}$ is a healthy spider.

Since $x_{d-2}$ and its children do not belong to the $\gamma_r(T)$-set $D$,
$x_{d-3}$ belongs to $D$.
Consider the tree $T':= T - T_{x_{d-2}}$.
It is easy to see that $V(T') \cap D$ is a $\gamma_r(T')$-set and $\gamma_{rI}(T') = 2\gamma_r(T')$.
Applying the inductive hypothesis to $T'$, we have $T' \in \mathcal{F}$.
Since $x_{d-3}$ is a leaf in $T'$, the tree $T$ can be obtained from the tree $T'$ by applying operation $\mathcal{O}_2$.
Thus, $T \in \mathcal{F}$.

\vskip5pt
Case 2.  $deg_T(x_{d-2})=2$.
Then $x_{d-3} \in D$.
Consider the tree $T':= T - T_{x_{d-2}}$
It is easy to see that $V(T') \cap D$ is a $\gamma_r(T')$-set and $\gamma_{rI}(T') = 2\gamma_r(T')$.
Applying the inductive hypothesis to $T'$, we have $T' \in \mathcal{F}$.
Since $x_{d-3}$ is a leaf in $T'$, the tree $T$ can be obtained from the tree $T'$ by applying operation $\mathcal{O}_1$.
Thus, $T \in \mathcal{F}$.
\end{proof}

\section{Open problems}\label{sec:open}

In this section, we discuss few open problems related to our results.
For any graph theoretical parameters $\sigma$ and $\delta$, we define a tree $T$ to be $(\sigma, \delta)$-tree if $\sigma(T) = \delta(T)$.
In general, it holds $\gamma_I(G) \leq \gamma_{rI}(G)$ for any graph $G$.
We suggest the following problem.

\begin{prob}
Characterize $(\gamma_I, \gamma_{rI})$-trees.
\end{prob}

In \cite{MCS}, D. Ma et al. gave the concept of total restrained domination.
Combining the properties of Italian dominating function and total restrained dominating set,
we give the concept of total restrained Italian dominating function, namely
a RIDF is a \textit{total restrained Italian dominating function} on $G$ if
the subgraph induced by $\{v \in V \mid f(v) \geq 1 \}$ has no isolated vertices.
We denote the total restrained Italian domination number by $\gamma_{rI}^t(G)$.
The total Italian domination number and total restrained domination number are
denoted by $\gamma_{I}^t(G)$ and $\gamma_r^t(G)$, respectively (see \cite{GMMY,MCS} for definitions).
We suggest the following problems.

\begin{prob}
Characterize $(\gamma_I^t, \gamma_{rI}^t)$-trees.
\end{prob}

\begin{prob}
Characterize $(\gamma_{rI}, \gamma_{rI}^t)$-trees.
\end{prob}

\begin{prob}
Characterize $(\gamma_r^t, \gamma_{rI}^t)$-trees.
\end{prob}

\bibstyle{plain}

\end{document}